\documentclass[12pt]{article}
\usepackage{amsmath,amsthm,amsfonts,amssymb,makeidx}
\def\ti{\tilde}
\theoremstyle{plain}
\newtheorem{thm}{Theorem}[section]
\newtheorem{lemma}[thm]{Lemma}

\newtheorem{prop}[thm]{Proposition}
\newtheorem{defn}[thm]{Definition}

\theoremstyle{definition}

\newtheorem{rem}[thm]{Remark}

\DeclareMathOperator{\Per}{Per}
\DeclareMathOperator{\Homeo}{Homeo}
\DeclareMathOperator{\Fix}{Fix}

\newcommand{\ha}{\hat \alpha}
\newcommand{\hb}{\hat \beta}

\newcommand{\A}{{\mathbb A}}

\newcommand{\F}{{\mathbb F}}
\newcommand{\cF}{{\mathcal F}}

\newcommand{\cR}{{\mathcal R}}
\newcommand{\R}{\mathbb R}

\newcommand{\Z}{\mathbb Z}

\makeindex
\author{John Franks}
\title{Rotation Numbers for $S^2$ diffeomorphisms}
\date{}
\begin{document}
 
\maketitle
\pagenumbering{arabic}

In these largely expository notes we describe the properties
of, and generalize, the function $\cR$ which assigns a number to a $4$-tuple
of distinct fixed points of an orientation
preserving homeomorphism or diffeomorphism of $S^2$. I am indebted to
Patrice LeCalvez for telling me many of the basic properties described
here.

\section{The function $\cR_f$ for fixed points.}

We let $f$ be an orientation preserving homeomorphism and 
denote by $X$ the set of fixed points of $f$. We will assume for
the moment that $X$ contains at least four distinct points.
Given four distinct points $x_1,x_2,x_3,x_4 \in \Fix(f)$ 
we consider the sets $A = \{x_1,x_2\}$ and $B = \{x_3,x_4\}.$
One can puncture at $A$ to obtain an annulus and then consider
a lift of $f$ to the universal cover of this annulus.  The difference
of the rotation numbers of the points of $B$ (with respect to this
lift) is an element of $\R$ which is independent of the 
choice of lift.  This gives a real valued function $\cR_f$ of the
four fixed points $\{x_1,x_2,x_3,x_4\}.$ This function has some
remarkable symmetries under permutations of its arguments which
form the content of this note.

We will be interested in the intersection of
oriented paths 
$\alpha: [0,1] \to S^2 \setminus B$
and $\beta: [0,1] \to S^2 \setminus A$ and 
with $\alpha$ running from  $x_1$ to $x_2,$ and 
$\beta$ running from $x_3$ to $x_4.$ These paths have an
algebraic intersection number determined by the orientation of
the paths and the orientation of $S^2$ which we now define.
The ends of $\beta,\ x_3$ and $x_4,$
are disjoint from $\alpha$ and  we can close $\beta$ to a 
loop by concatenating with a path from $\delta$ from $x_4$ to $x_3$ which lies
in the complement of the image of $\alpha.$ If $\delta * \beta$ is
loop obtained by concatenating $\delta$ and  $\beta$, then its
homology class $[\delta * \beta] \in H_1(S^2 \setminus A) \cong \Z$ 
is independent of the choice of $\delta.$  

\begin{defn}
We define the algebraic intersection number $\alpha \cdot \beta \in \Z$ by
\[
[\delta * \beta] =  (\alpha \cdot \beta) [u],
\]
where $u$ is an embedded closed loop in 
$S^2 \setminus A,$ positively oriented with respect to $x_1,$ 
and hence $[u]$ is a generator of $H_1(S^2 \setminus A)$.
\end{defn}

The number $\alpha \cdot \beta$ 
number depends only on the homology classes
$[\alpha] \in H_1(S^2 \setminus B, A)$ and
$[\beta] \in H_1(S^2 \setminus A, B).$ 
There is a standard skew-symmetric intersection pairing 
\[
i: H_1(S^2 \setminus B, A) \times H_1(S^2 \setminus A, B) 
\to H_0(S^2 \setminus A \cup B) \cong \Z.
\]
and the algebraic intersection number of $\alpha \cdot \beta$ is
equal to $i([\alpha], [\beta]).$  

There are three elementary facts we will have occasion to use
\begin{itemize}
\item The intersection is skew-symmetric: 
$\alpha \cdot \beta = - \beta \cdot \alpha.$
\item If $\hat \alpha(t) = \alpha(1-t)$, is the path with reverse
parametrization, then $\hat \alpha \cdot \beta = - \alpha \cdot \beta.$
\item If $g: S^2 \to S^2$ is an orientation preserving homeomorphism
then $(g \circ \alpha) \cdot (g \circ \beta) = \alpha \cdot \beta.$
\end{itemize}

\subsection{Characterization of $\cR_f.$}
In this section we give a definition of the function 
$\cR_f$ and provide several equivalent characterizations
(Propositions~(\ref{prop: cR T^n}), (\ref{prop: bhat-loop})
and (\ref{prop: gamma})). Throughout this section we will
assume that $f$ is an orientation preserving homeomorphism
and $X = \{x_1,x_2,x_3,x_4\}$ is a subset of $\Fix(f)$. 
The final result of this section shows that for a 
fixed $X$ the function $\cR_f$ is
a homomorphism from the group of orientation preserving
homeomorphisms which pointwise fix $X$ to $\Z$.

\begin{defn}\label{def: cR}
Let $f: S^2 \to S^2$ be an orientation preserving homeomorphism and
suppose $\{x_1,x_2,x_3,x_4\}$ is a subset of distinct points
of $\Fix(f).$  Let $A = \{x_1,x_2\}$ and $B = \{x_3,x_4\}$ and
let $\alpha$ be an oriented path in $S^2 \setminus B$
running from $x_1$ to $x_2$ and let $\beta$ 
be an oriented path in $S^2 \setminus A$ running from 
$x_3$ to $x_4.$ Then we define 
\[
\cR_{f}(x_1,x_2,x_3,x_4) = \alpha \cdot (f \circ \beta) - \alpha \cdot \beta.
\]
\end{defn}

The notation implies that 
$\cR_f(x_1,x_2,x_3,x_4)$ depends only on $x_1,x_2,x_3,x_4 \in \Fix(f)$
and not on the choice of $\alpha$ and $\beta$.  This is indeed true
and follows immediately from the following proposition which
gives an alternate description of $\cR_f.$ As in the definition
above we let $\{x_1,x_2,x_3,x_4\}$ be a set of distinct points in $\Fix(f)$
and define $A = \{x_1,x_2\}$ and $B = \{x_3,x_4\}.$

\begin{prop}\label{prop: cR T^n}
Let $V = S^2 \setminus A$ 
and let $\ti f : \ti V \to \ti V$ be the lift of $f$ to the 
universal covering space of $V$ which fixes $\ti x_3$, a lift
of $x_3$.  Let $T: \ti V \to \ti V$ be the generator of the 
group of covering translations corresponding to a loop in $V$
with intersection number $+1$ with some (and hence any) oriented path
from $x_1$ to $x_2$. If $\ti x_4 \in \ti V$ is a lift of $x_4,$ then
$\cR_f(x_1,x_2,x_3,x_4) = n \in \Z$ where $\ti f(\ti x_4) = T^n(\ti x_4).$
\end{prop}
\begin{proof}
Let $\ti x_3$ be a lift of $x_3$ and let $\ti f$ be the lift of 
$f$ which fixes $\ti x_3.$ 
If $\alpha$ and $\beta$ are as in Definition~(\ref{def: cR})
choose $\ti \alpha$ and $\ti \beta$ lifts of $\alpha$ and $\beta$
with $\ti \beta$ running from $\ti x_3$ to a point which is a 
lift of $x_4.$  We denote this point by $\ti x_4$ and define
$n$ by $\ti f(\ti x_4) = T^n(\ti x_4)$.  The value of $n$ would
be the same if we used any other lift of $x_4$ in place of
$\ti x_4.$  

The path $\ti f \ti \beta$ runs from $\ti x_3$ to 
$\ti f(\ti x_4) = T^n(\ti x_4)$.  It is homotopic to
the concatenation of the path $\ti \beta$ from 
$\ti x_3$ to  $\ti x_4$ with a path $\ti \gamma$ from $\ti x_4$ 
to $T^n(\ti x_4).$  If $\gamma$ is the closed loop in $V$ obtained
by projecting $\ti \gamma$ then $[\gamma]$ is $n$ times the 
generator of $H_1(V)$ which has intersection number $+1$ with
$\alpha.$  Hence 
\[
\alpha \cdot (f \circ \beta) =  \alpha \cdot \gamma + \alpha \cdot \beta
\]
and
\begin{equation}\label{beta-gamma}
\cR_f(x_1,x_2,x_3,x_4) = \alpha \cdot (f \circ \beta) 
- \alpha \cdot \beta =  \alpha \cdot \gamma = n.
\end{equation}
\end{proof}

\begin{rem}
We note that $\cR_f(x_1,x_2,x_3,x_4)$ is sometimes well defined
even if the points $x_1,x_2,x_3,$ and $x_4$ are not all distinct. In particular
if $x_1 = x_2$ but $\{x_2,x_3,x_4\}$ are distinct then
$\cR_f(x_2,x_2,x_3,x_4) = 0$ since we may choose the constant
path for $\alpha.$ Similarly if $x_3 = x_4$ but $\{x_1, x_2,x_3\}$ 
are distinct then
$\cR_f(x_1,x_2,x_3,x_3) = 0$ since we may choose the constant
path for $\beta.$  If $x_2 = x_3$ or $x_1 = x_4$ the value of
$\cR_f$ is undefined for the general homeomorphism $f$.
However, if $f$ is a $C^1$ diffeomorphism we can define
$\cR_f$ when $x_2 = x_3$ or $x_1 = x_4$, but the value 
may no longer be an integer, or even rational.  We will do this below in
Section~(\ref{sec: cR-diff}).
\end{rem}

Since $\beta$ and $f \circ \beta$ have the same endpoints,
$x_3$ and $x_4$, we can concatenate the paths $f \circ \beta$
and $\beta^- = \beta(1-t)$ to form a loop $\gamma$ in $S^2 \setminus A$
which also defines $\cR_f(x_1,x_2,x_3,x_4).$

\begin{prop} \label{prop: bhat-loop}
If $\beta^-(t) = \beta(1-t)$ and $\gamma = (f \circ \beta) * \beta^-$ then 
\[
\cR_f(x_1,x_2,x_3,x_4) = \alpha \cdot \gamma.
\]
\end{prop}

An easy consequence of Proposition~(\ref{prop: cR T^n})
is another useful description
of $\cR_f.$  We use that fact that the space of homeomorphisms
fixing three distinct points $\{x_1, x_2, x_3\}$ is contractible
and so there is an isotopy $f_t(x)$ $f_0 = id$ and $f_1 = f$ 
such that for all $t \in [0,1],\ f_t(x_i) = x_i$ for $1 \le i \le 3.$ 
This isotopy is unique up to homotopy as a path in 
the space of homeomorphisms fixing the points $\{x_1, x_2, x_3\}$.

\begin{prop}\label{prop: gamma}
Suppose $\alpha$ is a path from $x_1$ to $x_2$ which
is disjoint from the set $\{x_3,x_4\}.$ 
Let $\gamma_0$ be the closed loop in $S^2 \setminus \{x_1, x_2, x_3\}$
defined by $f_t(x_4)$ for $t \in [0,1]$ then
\[
\cR_{f}(x_1,x_2,x_3,x_4) = \alpha \cdot \gamma_0.
\]
\end{prop}

\begin{proof}
This follows immediately from the proof of 
Proposition~(\ref{prop: cR T^n}) because $\gamma_0$ lifts to 
a path $\ti \gamma_0$ in the universal covering 
space of $S^2 \setminus \{x_1,x_2\}$
with the same endpoints as the path $\ti \gamma$ in the 
proof of Proposition~(\ref{prop: cR T^n}). Hence loops $\gamma$
and $\gamma_0$ are homotopic in $S^2 \setminus \{x_1,x_2\}$
so from equation~(\ref{beta-gamma}) we have
\[
\cR_{f}(x_1,x_2,x_3,x_4) = \alpha \cdot \gamma = \alpha \cdot \gamma_0.
\]
\end{proof}

An important property of the function $\cR_f$ is that if $G$ is a subgroup
of $\Homeo(S^2)$ which  fixes a set $X$ pointwise then 
the function which assigns to $f \in G$ 
the function $\cR_f$ is a homomorphism.  For $f,g \in \Homeo(S^2)$
we will denote their composition by $fg.$

\begin{prop}\label{prop: group}
Suppose $\{x_1,x_2,x_3,x_4\}$ is a set of distinct points in 
$\Fix(f) \cap \Fix(g).$
Then
\[
\cR_{fg}(x_1,x_2,x_3,x_4) = \cR_f(x_1,x_2,x_3,x_4) 
+ \cR_g(x_1,x_2,x_3,x_4).
\]
\end{prop}

\begin{proof}
We use Proposition~(\ref{prop: gamma}).  As an isotopy from 
$id$ to $fg$ we choose $f_t g_t, \ t\in [0,1]$ where $f_t$ and
$g_t$ are isotopies from $id$ to $f$ and $g$ respectively, each
of which fixes the points $x_1,x_2,$ and $x_3.$
The loop $\gamma_0$ given by $f_tg_t(x_4)$ is homotopic to
the loop which is the concatenation of $\gamma_1$ given by
$g_t(x_4) = f_0g_t(x_4)$ with $\gamma_2$ given by
$f_t(x_4) = f_t(g_1(x_4)).$  
Hence 
\begin{align*}
\cR_{fg}(x_1,x_2,x_3,x_4) &= \alpha \cdot \gamma_0\\
&= \alpha \cdot \gamma_1 + \alpha \cdot \gamma_2\\
&= \cR_g(x_1,x_2,x_3,x_4)+ \cR_f(x_1,x_2,x_3,x_4).
\end{align*}
\end{proof}

\subsection{Basic Symmetries of $\cR_f$.}
We want to understand the effect permuting the variables
of $\cR_f.$ We first establish  basic relations --
equations (\ref{cR-tau}), (\ref{cR-sigma}),
(\ref{coboundary12}), and (\ref{coboundary34}), below.
In Section~(\ref{sec: perm}) we establish relations
applicable to any permutation of the variables of $\cR_f.$

We will write elements of
$S_4$, the symmetric group of order four,
in the standard way as a product of cycles.
For example $(12)(34)$ is the product the transposition
of $1$ and $2$ and the transposition of $3$ and $4$.
Recall that the group $S_4$ is generated by the three
transpositions $\sigma_i = (i,i+1),\ 1 \le i \le 3$.

If $x = (x_1,x_2,x_3,x_4)$ is a four-tuple of elements of $\Fix(f)$
we will denote the four-tuple 
$(x_{\sigma(1)}, x_{\sigma(2)}, x_{\sigma(3)}, x_{\sigma(4)})$
by $x_\sigma.$
An element of $S_4$ which we will have occasion to use several times is the
cyclic permutation $(123)$. We will denote this element by $\tau.$

\begin{prop}\label{prop: cR-sym}
Suppose $\{x_1,x_2,x_3,x_4\}$ is a set of distinct points in $\Fix(f).$
Then
\begin{gather}
\cR_f(x) + \cR_f(x_\tau) + \cR_f(x_{\tau^2}) = 0, \label{cR-tau} \\
\cR_f(x_{\sigma_1}) = \cR_f(x_{\sigma_3}) = -\cR_f(x) \label{cR-sigma}
\end{gather}
\end{prop}

\begin{proof}
Equation~(\ref{cR-sigma}) follows immediately from the definition
of $\cR_f$ since reversing the parametrization of $\alpha$
to get $\bar \alpha$, a path from $x_2$ to $x_1$ changes the
sign of the intersection number.  Hence
\begin{align*}
\cR_f(x_{\sigma_1}) &= \bar \alpha \cdot (f \circ \beta) 
-\bar \alpha \cdot \beta\\
&= - \alpha \cdot (f \circ \beta) + \alpha \cdot \beta\\
&= - \cR_f(x).
\end{align*}
The argument for $\sigma_3$ is similar.

To show equation~(\ref{cR-tau}) we choose paths 
$\alpha_1$ from $x_1$ to $x_2$,\ $\alpha_2$ from $x_2$ to $x_3$, 
and $\alpha_3$ from $x_3$ to $x_1.$  The concatenation of
these three paths is a closed loop $\eta$ in $S^2.$ Let $\gamma$
be as defined in Proposition~(\ref{prop: gamma}). Then
\[
\cR_f(x) + \cR_f(x_\tau) + \cR_f(x_{\tau^2}) 
= \alpha_1 \cdot \gamma + \alpha_2 \cdot \gamma + \alpha_3 \cdot \gamma
= \eta \cdot \gamma =0,
\]
since the intersection number of any two closed loops in $S^2$ 
is $0.$
\end{proof}

\begin{prop}\label{prop: cR-sym2}
Suppose $\{x_1,x_2,x_3,x_4,w\}$ is a set of distinct points in $\Fix(f).$
Then
\begin{align}
\cR_f(x_1,x_2,x_3,x_4) &= \cR_f(x_3,x_4,x_1,x_2), \label{(13)(24)}\\
\cR_f(x_1,x_2,x_3,x_4) &= \cR_f(x_1,w,x_3,x_4) + \cR_f(w,x_2,x_3,x_4),
\text{ \quad and }\label{coboundary12}\\
\cR_f(x_1,x_2,x_3,x_4) &= \cR_f(x_1,x_2,x_3,w) + \cR_f(x_1,x_2,w,x_4).
\label{coboundary34}
\end{align}
\end{prop}

\begin{proof}
To show equation~(\ref{(13)(24)}) we note that Proposition~(\ref{prop: group})
implies $\cR_f = - \cR_{f^{-1}}.$  Let $\alpha$ (respectively $\beta$) be
a path from $x_1$ to $x_2$ (respectively $x_3$ to $x_4$) which
we choose so that $\alpha \cdot \beta = 0.$  Then
\begin{align*}
\cR_f(x_1,x_2,x_3,x_4) &= \alpha \cdot (f \circ \beta)\\
&= - (f \circ \beta) \cdot \alpha \text{\quad since $\ \cdot\ $ 
is skew-symmetric,}\\
&= - \beta \cdot (f^{-1} \circ \alpha)\\
&= - \cR_{f^{-1}}(x_3,x_4,x_1,x_2)\\
&= \cR_{f}(x_3,x_4,x_1,x_2).
\end{align*}

To show equation~(\ref{coboundary12}) we
choose paths $\alpha_1$ from $x_1$ to $w$ and $\alpha_2$ from $w$ to $x_2$ 
and let $\alpha$ be the path from $x_1$ to $x_2$ obtained by
concatenating them.  Then
\begin{align*}
\cR_f(x_1,x_2,x_3,x_4) &= \alpha \cdot (f \circ \beta) - \alpha \cdot \beta\\
&= \alpha_1 \cdot (f \circ \beta) - \alpha_1 \cdot \beta 
+ \alpha_2 \cdot (f \circ \beta) - \alpha_2 \cdot \beta\\
&= \cR_f(x_1,w,x_3,x_4) + \cR_f(w,x_2,x_3,x_4).
\end{align*}
Equation~(\ref{coboundary34}) follows from equations~(\ref{(13)(24)})
and (\ref{coboundary12}).
\end{proof}

\subsection{$\cR_f$ for diffeomorphisms} \label{sec: cR-diff}

For a homeomorphism $f$ the value of $\cR_f(x_1,x_2,x_3,x_4)$ is
generally undefined if either of $x_1,x_2$ is equal to either
of $x_3,x_4.$ In the case of $C^1$ diffeomorphisms we can remove
this restriction.  This done essentially by incorporating the
infinitesimal rotation number at the fixed point, say $x_1 = x_3$,
in question.

Suppose $p \in \Fix(f)$. Let $\alpha: [0,1] \to S^2$ and 
$\beta: [0,1] \to S^2$ be
$C^1$ embeddings with $\alpha(0) = \beta(0) = p$ with
distinct tangent vectors at $p.$  We can blow up the point $p$ to obtain
a map $\hat f$ on the compactification of $S^2 \setminus \{p\}.$
The action of $\hat f$ on the circle which compactifies $S^2 \setminus \{p\}$
is the projectivization of $Df_p$. We denote the compactified 
space (which is topologically a closed disk) by $M.$
The points on the boundary of $M$ which correspond to the tangents
to $\alpha$ and $\beta$ at $p$ are distinct and the ends of 
paths $\ha$ and $\hb$ which agree with $\alpha$ and $\beta$
on $(0,1].$
We define $\alpha \cdot \beta$ to be $\ha \cdot \hb.$

We can now define $\cR_f(p, x_2, p, x_4)$.  Choose embeddings
$\alpha$ and $\beta$ satisfying 
\begin{enumerate}
\item[(1)]  $\alpha(0) = \beta(0) = p \in \Fix(f).$ 
\item[(2)] $\alpha(1) = x_2, \beta(1) = x_4$ with $x_1, x_2 \in \Fix(f)$
\item[(3)] If $v$ and $w$ are the tangents at $p$
to $\alpha$ and $\beta$ respectively then
\[
\frac{df_p^n(v)}{\|df_p^n(v)\|} \ne \frac{w}{\|w\|},
\]
for all $n \in \Z.$
\end{enumerate}

\begin{lemma}
Suppose the paths $\alpha$ and $\beta$ satisfy (1) - (3) above.
Then the limit 
\[
= \lim_{n \to \infty} \frac{\alpha \cdot (f^n \circ \beta)}{n}
\]
exists and is independent of the choice of $\alpha$ and $\beta$.
\end{lemma}

\begin{proof}
Blow up the points $p$ and $x_2$ to form a closed
annulus $\A.$ Choose the lift $\ti f :\ti \A \to \ti \A$ which
fixes a lift of $x_4.$ 
Let $C$ be the circle added when $p$ was blown up; so $C$ is 
one component of the boundary
of $\A$. Parametrize $C$ so that $\ha(0)$ is $0$ and let 
$\ti C$ be the universal cover of $C$ which we may think of as
a component of the boundary of $\ti A.$ 

Then $\ti f|_{\ti C}$ is a lift of $f|_C$ and it has a well defined
translation number $\tau(\ti f)$ defined to be
\[
\tau(\ti f) = \lim_{n \to \infty} \frac{ \ti f^n(\ti p_2) - \ti p_2}{n}
\]
where $\ti p_2$ is the lift of the endpoint $\beta(0).$  This limit
always exists by the standard theory of rotation numbers for circle
homeomorphisms.  The annulus $\A$ depends only on $p$ and $x_2$.
The lift $\ti f$ is determined by $x_4.$ 

Finally we note that 
\[
\big (\alpha \cdot (f^n \circ \beta) - \alpha \cdot \beta \big)
= - \big (\ti f^n(\ti p_2) - \ti p_2 \big ) +K,
\]
where $K$ satisfies $|K| \le 2.$
It follows that 
\[
\lim_{n \to \infty} 
\frac{\alpha \cdot (f^n \circ \beta) - \alpha \cdot \beta}{n} = -\tau(\ti f).
\]
\end{proof}

\begin{defn}
We define 
\[
\cR_f(p, x_2, p, x_4) 
= \lim_{n \to \infty} 
\frac{\alpha \cdot (f^n \circ \beta) - \alpha \cdot \beta}{n}.
\]
\end{defn}

The number $\cR_f(x_1, p, x_3, p)$ is defined analogously
as are the values of $\cR_f$ when one of first two and
one of its second two variables are equal to $p.$
We also observe that if $p_1, p_2 \in \Fix(f)$ then 
$\cR_f(p_1, p_2, p_1, p_2)$ is also defined and equal
to the difference of the rotation numbers of $\hat f$
on the two boundary components of the annulus $\A$
obtained by blowing up both $p_1$ and $p_2.$

It is easy to check that $\cR_f$ as defined still satisfies
the symmetries of Propositions~(\ref{prop: cR-sym2}) and 
(\ref{prop: cR-sym2}).

\subsection{$\cR_f$ for periodic points}

From Proposition~(\ref{prop: group}) it is clear that
for $q \in \Z$ we have $\cR_{f^q} = q \cR_f$. It follows
that if $\{x_1,x_2,x_3,x_4\}$ is a subset of distinct
points in $\Fix(f)$, then
\[
\cR_{f}(x_1,x_2,x_3,x_4) = \frac{1}{q}\cR_{f^q}(x_1,x_2,x_3,x_4).
\]

\begin{defn} Suppose $\{x_1,x_2,x_3,x_4\}$ is a subset of distinct
points in $\Fix(f^q)$. We define
\[
\cR_{f}(x_1,x_2,x_3,x_4) = \frac{1}{q}\cR_{f^q}(x_1,x_2,x_3,x_4).
\]
\end{defn}

Because of the linearity of this definition it is clear that
Propositions (\ref{prop: group}), (\ref{prop: cR-sym}), and 
(\ref{prop: cR-sym2}) all remain valid when the argument
of $\cR_f$ is $x = (x_1,x_2,x_3,x_4),$ a $4$-tuple of distinct points
in $\Per(f)$, the set of periodic points of $f.$



\section{The complete symmetries of $\cR_f.$}\label{sec: perm}
If $X$ is a set we want to investigate the class of functions of 
four variables 
$\cF: X^4 \to \R$ which possess certain symmetries under permutation
of the variables.  The simplest example with the symmetries
which interest us occurs when $X = \R$ and
we define the quadratic function $q: \R^4 \to \R$
by $q(x_1,x_2,x_3,x_4) = (x_1-x_2)(x_3-x_4).$

As before we will write elements of
$S_4$ in the standard way as a product of cycles and we
recall that $\sigma_i$ denotes the transposition $(i,i+1),\ 1 \le i \le 3$.
The group $S_4$ is generated by these three transpositions 

If $x = (x_1,x_2,x_3,x_4)$ is a four-tuple of elements of some
set we will denote the four-tuple 
$(x_{\sigma(1)}, x_{\sigma(2)}, x_{\sigma(3)}, x_{\sigma(4)})$
by $x_\sigma.$
As before an element of $S_4$ which we will denote the
cyclic permutation $(123)$ by $\tau.$

We are interested in the vector space of all
functions $\cF:X^4 \to \R$ which satisfy the 
following three equations.
\begin{gather}
\cF(x) + \cF(x_\tau) + \cF(x_{\tau^2}) = 0, \label{tau} \\
\cF(x_{\sigma_1}) = \cF(x_{\sigma_3}) = -\cF(x) \label{sigma}\\
\cF(x_1,x_2,x_3,x_4) = \cF(x_1,w,x_3,x_4) + \cF(w,x_2,x_3,x_4),
\text{\quad for all $w \in X$}.\label{coboundary}
\end{gather}

Equation~(\ref{tau}) says
that if we cyclically permute the first three arguments of $\cF$
and sum the values, the result is $0.$  This is a kind of Jacobi
identity for $\cF$.  Equation~(\ref{sigma}) says
that $\cF$ is skew-symmetric in its first two variables and
in its last two.
Equation~(\ref{coboundary}) says that as a function of its
first two variables $\cF$ is a coboundary.
We remark that it will be shown below (see Remark~(\ref{remark}))
that a consequence
of these relations is that $\cF(x_1,x_2,x_3,x_4) = \cF(x_3,x_4,x_1,x_2)$
so it follows that  $\cF$ is also a coboundary in its last two
variables, i.e., for all $w \in X,$
\[
\cF(x_1,x_2,x_3,x_4) = \cF(x_1,x_2,x_3,w) + \cF(x_1,x_2,w,x_4).
\]

To understand the  symmetries of $\cF$ is useful to 
introduce a function to $\R^3$ containing the three
values obtained by cyclically permuting the first
three arguments of $\cF.$
Hence we define the function $\F: X^4 \to \R^3$ by
\begin{equation}
\F(x) = \big ( \cF(x), \cF(x_\tau), \cF(x_{\tau^2}) \big )
\end{equation}

An immediate
consequence of equation~(\ref{tau}) is the fact that
the image of $\F$ lies in the subspace of $\R^3$
given by $y_1 + y_2 + y_3 = 0.$

The symmetries of $\cF$ under the action of $S_4$ permuting
its arguments are most easily described by means of a representation
of $S_4$ in $GL(3,\Z).$
Since the group $S_4$ is generated by the three
transpositions $\sigma_i = (i,i+1),\ 1 \le i \le 3$, we can
define a representation $\Theta$ by specifying its value on these three
elements of $S_4$. We define
\begin{gather*}
\Theta(\sigma_1) = \Theta(\sigma_3) = 
\begin{bmatrix}
-1 & 0 & 0\\
0 & 0 & -1\\
0 & -1 & 0\\
\end{bmatrix},\\
\Theta(\sigma_2) = 
\begin{bmatrix}
0 & 0 & -1\\
0 & -1 & 0\\
-1 & 0 & 0\\
\end{bmatrix}.
\end{gather*}

\begin{prop}
The function $\Theta$ defined for $\sigma_i,$\ $1 \le i \le 3,$
determines a homomorphism from $\Theta: S_4 \to GL(3,\Z).$
\end{prop}

\begin{proof}
The set $\sigma_1,\sigma_2,\sigma_3 \in S_4$ is a set of
generators.
It is easily checked that $\Theta$ respects the relations
$\sigma_i^2 = id,\ \sigma_1 \sigma_3 = \sigma_3 \sigma_1$
and $\sigma_i \sigma_{i+1} \sigma_i
= \sigma_{i+1} \sigma_i \sigma_{i+1},$ which are
a complete set of relations defining $S_4$ so 
$\Theta$ is a homomorphism.  
\end{proof}

\begin{rem}\label{remark}
Since the transposition $(13) = \sigma_1 \sigma_2 \sigma_1$ and
$(24) = \sigma_2 \sigma_3 \sigma_2$ it is easy to calculate
\[
\Theta((13)) = \Theta((24)) = 
\begin{bmatrix}
0 & -1 & 0\\
-1 & 0 & 0\\
0 & 0 & -1\\
\end{bmatrix},
\]
and hence that $\Theta((13)(24)) = I.$
It is easy to check that the kernel of $\Theta$
is a group isomorphic to $\Z/2\Z \oplus \Z/2\Z$ which is generated
by the elements $(12)(34)$ and $(13)(24)$ of $S_4.$  One can
also check that the image of $\Theta$ is isomorphic to $S_3.$
\end{rem}

\begin{thm}\label{F-symmetry}
Suppose for all $x \in X^4$ the function $\cF: X^4 \to \R$ satisfies
equations (\ref{tau}) and (\ref{sigma}) and we define $\F: X^4 \to \R^3$
by
\[
\F(x) = \big ( \cF(x), \cF(x_\tau), \cF(x_{\tau^2}) \big ).
\]
Then for every $\sigma \in S_4$ and every $x \in X^4$
\begin{equation}\label{F-theta}
\F(x_{\sigma})
= \Theta(\sigma) \F(x).
\end{equation}
\end{thm}

We postpone the proof of this result to the next section.
Note that this result about the symmetries of $\cF$ and $\F$ requires
only that $\cF$ satisfy equations (\ref{tau}) and (\ref{sigma}).
For the next result
we will additionally use the coboundary condition, equation~(\ref{coboundary}).
We note that given {\em any} function of two variables
$g: X^2 \to \R$ if we define 
\[
\cF(x)  = g(x_1,x_3) - g(x_1,x_4)-  g(x_2,x_3) + g(x_2,x_4),
\]
then it is straightforward to check that $\cF$ satisfies 
equations (\ref{tau}) and (\ref{sigma}). The next theorem asserts that
the converse of this is also true.

\begin{thm}\label{thm: g-function}
Suppose for all $x \in X^4$ the function $\cF: X^4 \to \R$ satisfies
equations (\ref{sigma}), and (\ref{coboundary})
Then there exists a function $g: X^2 \to \R$ such that
\begin{equation}
\cF(x_1,x_2,x_3,x_4)  = g(x_1,x_3) - g(x_1,x_4) 
-g(x_2,x_3) + g(x_2,x_4). \label{eqn: g}
\end{equation}
If we normalize by picking two distinguished elements $a,b \in X$ and 
specifying that $g(a,w) = g(w,b) = 0$ for all $w \in X$ then
$g$ is unique.
\end{thm}

\begin{proof}
Given $\cF$ satisfying equations equations (\ref{sigma})
and (\ref{coboundary}), and two distinguished
elements $a,b \in X$ we define
\[
g(u,v) = \cF(u,a,v,b).
\]
Skew-symmetry in the first two variables and last two variables
implies $\cF(a,a,v,b) = 0$ and
$\cF(u,a,b,b) = 0$, so $g(a,w) = g(w,b) = 0$ for all $w \in X.$

To show that equation~(\ref{eqn: g}) is satisfied we make repeated
use of the skew-symmetry in the first two or last two variables
(equation~(\ref{sigma})) and of the coboundary relation in the 
first two and last two variables (equation~(\ref{coboundary})).

\begin{align*}
\cF(x_1,x_2,x_3,x_4)  &= \cF(x_1,a,x_3,x_4)  + \cF(a,x_2,x_3,x_4) \\
&= \cF(x_1,a,x_3,x_4)  - \cF(x_2, a, x_3,x_4)\\
&= \big( \cF(x_1,a,x_3,b)   + \cF(x_1,a,b,x_4)\big)
- \big( \cF(x_2, a, x_3 ,b) + \cF(x_2, a, b,x_4)\big)\\
&= \cF(x_1,a,x_3,b) - \cF(x_1,a,x_4,b)
- \cF(x_2, a, x_3 ,b) + \cF(x_2, a, x_4, b)\\
&= g(x_1,x_3) - g(x_1,x_4) 
-g(x_2,x_3) + g(x_2,x_4).
\end{align*}

To show uniqueness, suppose $g_1: X^2 \to \R$ is another function
satisfying the conclusion of the theorem.  Then for any $u,v \in X$
\begin{align*}
g(u,v) & = \cF(u,a,v,b) \\
&= g_1(u,v) - g_1(u,b)-  g_1(a,v) + g_1(a,b)\\
&= g_1(u,v).
\end{align*}

\end{proof}

\section{Proof of Theorem~(\ref{F-symmetry})}

{\bf Theorem~(\ref{F-symmetry})}
{\em Suppose for all $x \in X^4$ the function $\cF: X^4 \to \R$ satisfies
equations (\ref{tau}) and (\ref{sigma}) and we define $\F: X^4 \to \R^3$
by
\[
\F(x) = \big ( \cF(x), \cF(x_\tau), \cF(x_{\tau^2}) \big ).
\]
Then for every $\sigma \in S_4$ and every $x \in X^4$
\begin{equation}
\F(x_{\sigma})
= \Theta(\sigma) \F(x).
\end{equation}
}

\begin{proof} 
Since $\Theta$ is a homomorphism it suffices to prove that
$\F(x_{\sigma}) = \Theta(\sigma) \F(x)$ for each $\sigma$
in the set of generators of $\{\sigma_1,\sigma_2, \sigma_3\}.$

Let $x = (a,b,c,d).$ We define the numbers $r$ and $s$ by
$\cF(a,b,c,d) = r$ and $\cF(b,c,a,d) = s.$
Then
\begin{equation}\label{F-tau}
\begin{bmatrix}
\cF(x)\\  
\cF(x_{\tau})\\
\cF(x_{\tau^2})
\end{bmatrix}
= \begin{bmatrix}
\cF(a,b,c,d)\\
\cF(b,c,a,d)\\
\cF(c,a,b,d)
\end{bmatrix}
= \begin{bmatrix}
r\\
s\\
-r-s
\end{bmatrix}
\end{equation}
where the equality of the third component 
follows from equation~(\ref{tau}).

We next conclude 
\begin{equation}\label{F-sigma}
\begin{bmatrix}
\cF(x_{\sigma_1})\\  
\cF(x_{\tau \sigma_1})\\
\cF(x_{\tau^2 \sigma_1})
\end{bmatrix}
= \begin{bmatrix}
\cF(b,a,c,d)\\
\cF(a,c,b,d)\\
\cF(c,b,a,d)
\end{bmatrix}
= \begin{bmatrix}
-r\\
r+s\\
-s
\end{bmatrix}.
\end{equation}
where the first component comes from equation~(\ref{sigma}),
the third component comes from equation~(\ref{sigma}) applied
to the second component of equation~(\ref{F-tau}) 
and the second component follows from equation~(\ref{tau}).

Since
\[
\begin{bmatrix}
-1 & 0 & 0\\
0 & 0 & -1\\
0 & -1 & 0
\end{bmatrix}
\begin{bmatrix}
r\\
s\\
-r-s
\end{bmatrix}
= 
\begin{bmatrix}
-r\\
r+s\\
-s
\end{bmatrix}
\]
we have shown $\F(x_{\sigma_1}) = \Theta(\sigma_1) \F(x).$

Likewise from equation~(\ref{F-sigma}) we conclude
\begin{equation}
\begin{bmatrix}
\cF(x_{\sigma_2})\\  
\cF(x_{\tau \sigma_2})\\
\cF(x_{\tau^2 \sigma_2})
\end{bmatrix}
= \begin{bmatrix}
\cF(a,c,b,d)\\
\cF(c,b,a,d)\\
\cF(b,a,c,d)
\end{bmatrix}
= \begin{bmatrix}
r+s\\
-s\\
-r
\end{bmatrix}.
\end{equation}
Hence $\F(x_{\sigma_2}) = \Theta(\sigma_2) \F(x).$

To handle the case of $\sigma_3$ we define
$t = -\cF(x_{\tau^2\sigma_3}) = -\cF(d,a,b,c)$ so
\begin{equation}\label{F-sigma3}
\begin{bmatrix}
\cF(x_{\sigma_3})\\  
\cF(x_{\tau \sigma_3})\\
\cF(x_{\tau^2 \sigma_3})
\end{bmatrix}
= \begin{bmatrix}
\cF(a,b,d,c)\\
\cF(b,d,a,c)\\
\cF(d,a,b,c)
\end{bmatrix}
= \begin{bmatrix}
-r\\
r+t\\
-t
\end{bmatrix}
\end{equation}
where equality of the first components follows by equation~(\ref{sigma}) and
the value of the second component again comes from 
the fact that the sum of the components is $0$ (equation ~(\ref{tau})).

From the last component of 
equation~(\ref{F-sigma}) we get $\cF(c,b,d,a) = s$ and 
the second component of (\ref{F-sigma3}) shows
$\cF(b,d,c,a) = -r-t.$ 
By equation~(\ref{tau}), 
\[
\cF(c,b,d,a)+ \cF(b,d,c,a) + \cF(d,c,b,a) = 0.
\]
and hence
\[
s +(-r -t) +\cF(d,c,b,a) = 0.
\]
So $\cF(d,c,b,a) = r+t -s$ and $\cF(d,c,a,b) = -r -t +s.$ 
Also $\cF(c,a,d,b) = r+s$ by the third component of equation~(\ref{F-tau}).
Using equation ~(\ref{tau}) once again, we see
\[
\cF(c,a,d,b)+ \cF(a,d,c,b) + \cF(d,c,a,b) = 0,
\]
and hence
\[
(r + s) + \cF(a,d,c,b) + (-r -t +s) = 0.
\]
So $\cF(a,d,c,b) = -2s+t.$
But  the third component of equation~(\ref{F-sigma3}) implies
$\cF(a,d,c,b)$ is also equal to $-t$, so we conclude that 
$-t =  -2s+t$ and hence $s = t.$
Substituting $s$ for $t$ in equation~(\ref{F-sigma3}) gives
$\F(x_{\sigma_3}) = \Theta(\sigma_3) \F(x).$
\end{proof}

\end{document}